\documentclass[a4paper,11pt]{amsart}
\usepackage[utf8x]{inputenc}
\usepackage{amsthm,amsfonts,amsmath,amssymb,enumerate}
\usepackage{verbatim} 
\usepackage{rotating} 

\usepackage{hyperref}
\usepackage{pdfsync}
\usepackage{cancel}

\usepackage{multirow}
\usepackage{color}

\usepackage{tikz}

\newcommand{\N}{\mathcal{N}}

\newcommand{\C}{\mathbb{C}}

\newcommand{\g}{\mathfrak{g}}
\renewcommand{\v}{\mathfrak{v}}
\newcommand{\z}{\mathfrak{z}}
\newcommand{\n}{\mathfrak{n}}
\newcommand{\h}{\mathfrak{h}}
\renewcommand{\a}{\mathfrak{a}}

\newtheorem*{thmA}{Theorem A}
\newtheorem*{thmB}{Theorem B}
\newtheorem*{thmC}{Theorem C}
\newtheorem*{thmD}{Theorem D}

\newtheorem{theorem}{Theorem}[section]
\newtheorem{lemma}[theorem]{Lemma}

\theoremstyle{definition}

\theoremstyle{remark}
\newtheorem{remark}[theorem]{Remark}

\numberwithin{equation}{section}

\begin{document}

\title[]{Rigid 2-step graph Lie algebras}
\author{Josefina Barrionuevo and Paulo Tirao}
\address{CIEM-FaMAF, Universidad Nacional de Córdoba, Argentina}
\date{February, 2022}

\begin{abstract}
 We characterize those graphs which correspond to a rigid 2-step nilpotent
 Lie algebra in the variety of at most 2-step nilpotent Lie algebras.
\end{abstract}

\maketitle

In this work we assume that the field is $\C$ and all Lie brackets are complex Lie brackets. However some results are valid for fields of caracteristic $0$ or infinite fields. 

Understanding which Lie algebras are rigid within different varieties of Lie algebras is a deep and difficult problem.
Already finding some rigid Lie algebras is quite challenging.
On the other hand showing that large families are not rigid sheds some light  on the search for the rigid ones.

Let $\N_{n,2}$ be the variety of complex nilpotent Lie algebras of dimension $n$ which are at most 2-step nilpotent, 
that is either abelian or 2-step nilpotent.
On the one hand there are several cohomological characterizations of rigidity in $\N_{n,2}$ or 2-rigidity \cite{A, BT, GR}.
On the other hand the family of \emph{2-step graph Lie algebras}
have attracted particular attention in different frameworks, 
due in part to its strong combinatorial flavor.

The free 2-step nilpotent Lie algebras correspond to the complete graph.
Recently we proved that all free $k$-step nilpotent Lie algebras are rigid in the variety $\N_{n,k}$, of complex nilpotent Lie algebras at most $k$-step nilpotent, and not longer rigid in $\N_{n,k+1}$ \cite{BT}. 
That is they are $k$-rigid but not $(k+1)$-rigid. 

Notice that each isolated vertex of the underlying graph corresponds to an abelian 1-dimensional factor of the graph
Lie algbra. These graph algebras are, in general, not $2$-rigid \cite{BT}.

In this short note we prove that besides the free 2-step nilpotent Lie 
algebras there are only few more 2-rigid graph Lie algebras all of which correspond to graphs with at most 4 vertices and are 
of dimension less than or equal to 8.

\section{2-rigidity}

We say that a Lie bracket $\mu\in\N_{n,2}$ is $2$-rigid if $\mu$ 
is rigid in the variety $\N_{n,2}$, that is its $GL_n(\C)$-orbit
is Zarisky open in $\N_{n,2}$.

For the purpose of this paper, to prove rigidity we rely
on \cite{BT} and its Appendix.
The results that we will use, stated for $\N_{n,2}$, are the following:

\begin{thmA}[Sulca - Corollary 7.4 in \cite{BT}]\label{thm:A}
 Let $\mu\in\N_{n,2}$. If $H^2_{2\textit{-nil}}(\mu,\mu)=0$, then $\mu$ is $2$-rigid. 
\end{thmA}

\begin{thmB}[Theorem 4.3 in \cite{BT}]\label{thm:B}
 The free 2-step nilpotent Lie algebra on $m$ generators is 2-rigid.
\end{thmB}

\begin{thmC}[Theorem 6.10 in \cite{BT}]\label{thm:C}
 If $\a_l$ is the $l$-dimensional abelian bracket, with $l\ge 1$, then 
 $\n\oplus\a_l\in\N_{n,2}$ is 2-rigid if and only if $\n=\h_1$ and $l=1$.
\end{thmC}

To prove non-rigidity within the class of 2-step graph Lie algebras, 
Theorem D below suffices. 
To prove it we use some preliminary result, that can be found in
\cite{A}.

Recall that if $\mu\in\N_{n,2}$, we can descompose $\C^n=\v\oplus\z$ 
with $\z$ the center of $\mu$ and $\v$ any direct complement of it.

\begin{lemma}[Alvarez]
Let $\mu\in\N_{n,2}$, with $\C^n=\v\oplus\z$. If $\mu$ is $2$-rigid, then $\Lambda^2\v^*\otimes\z\subseteq B^2(\mu,\mu)$. 
\end{lemma}

\begin{proof}
 If $\varphi \in \Lambda^2\v^*\otimes\z$, then 
 $\mu_t=\mu+t\varphi$ is a $2$-step analytic deformation and then is equivalent to the trivial deformation of $\mu$ \cite[Theorem 7.1]{M}.
 [In characteristic zero, the analytic and geometric rigidity are
 equivalent concepts.] This implies that
 $\varphi\in B^2(\mu,\mu)$. 
\end{proof}

\begin{remark} In \cite{A} there is also the proof of the converse of the previous lemma (that is not necessary in this paper).
\end{remark}

\begin{thmD}\label{thm:D}
Let $\mu\in\N_{n,2}$, with $\C^n=\v\oplus\z$. If there exist $v,w\in\v$ linearly independent such that, $[v,w]=0$ and $\langle [v,\g]\cup [w,\g]\rangle \subsetneq \z$, then $\mu$ is not $2$-rigid.
\end{thmD}

\begin{proof}
Choose $z\in \z-\langle [v,\g]\cup [w,\g]\rangle$ and consider
$\varphi=v^*\wedge w^*\otimes z$.
On the one hand
$\varphi\in \Lambda^2\v\otimes \z$. 
On the other hand $\varphi\notin B^2(\mu,\mu)$.
In fact, if $\varphi=\delta^1(f)$ for a linear map $f:\g\rightarrow\g$,
then
\begin{eqnarray*}
(v^*\wedge w^*\otimes z)(v,w) &=& \delta^1f(v,w) \\
z &=& [f(v),w]+[v,f(w)]-f([v,w]) \\
z &=& [f(v),w]+[v,f(w)]
\end{eqnarray*}
and hence $z\in \langle [v,\g]\cup [w,\g]\rangle$, wich is an absurd since we started this proof choosing $z\in \z-\langle [v,\g]\cup [w,\g]\rangle$. 

Therefore $\Lambda^2\v\otimes\z\not\subseteq B^2(\mu,\mu)$
and since the previous lemma $\mu$ is not $2$-rigid.  
\end{proof}

\section{Graph Lie algebras}

Let $G=(V,A)$ be a graph with vertices $V=\{v_1,\dots,v_m\}$
and edges $A\subseteq \{a_{ij}:\, 1\le i<j\le m\}$. 

The graph Lie algebra associated to $G$, $\g_G$, is the 
$\C$-vector space generated by $V\cup A$ and Lie bracket defined by
\[
[v_i,v_j]= \begin{cases} 
a_{ij}, &\text{\ if\ } a_{ij}\in A, \\
0, &\text{otherwise}, 
\end{cases}
\]
for $1\le i<j\le m$.
If $A=\emptyset$, then $\g_G$ is the $m$-dimensional abelian Lie algebra, but in general if $A\neq \emptyset$, $\g_G$ is a $2$-step nilpotent Lie algebra. 

Each connected component of $G$ corresponds to an ideal of $\g_G$ 
and $\g_G$ is the direct sum of these ideals.
Isolated vertices are 1-dimensional ideals and the sum
of all of them is a maximal abelian factor of $\g_G$.

We are interested on their rigidity or non rigidity 
inside $\N_{n,2}$.

\section{2-rigidity of graph Lie algebras}

Before going to the general case, let us consider the cases
$m\le 4$.
The corresponding algebras are of small dimension, less than or equal to 10, and many have an abelian factor.
For each one of these there are more than one way to justify
its 2-rigidity or its non 2-rigidity.
To prove 2-rigidity we use the theorems A and B.
And for non 2-rigidy we use the theorems C and D.

\vskip1cm

\begin{remark}
To prove $2$-rigidity we distinguish two cases, $2$-step free (Theorem B) and others (Theorem A). 
To prove non $2$-rigidity we will distinguish two cases, with abelian factors (Theorem C) or without (Theorem D). In the last case, the descomposition will be $\C^n=\v\oplus\z=\langle V\rangle\oplus \langle A\rangle$.
\end{remark}

\begin{remark}\label{rmk:3-rigid}
	Every graph Lie algebra, but $\h_1$ (Heisenberg) and $\a_1$
	and $a_2$ (abelian of dimensions 1 and 2), can be deformed to a
	3-step nilpotent Lie algebra \cite [Section 3]{BT}. 
	So that graph Lie algebras, up to these 3 exceptions, 
	are never 3-rigid.
\end{remark}

\begin{center}
Graph Lie algebras with 2 or 3 vertices \\
\ \\
 \begin{tabular}{|c|c|c||c|c|c|}
 	\hline
 \raisebox{-0.5\height}{ 
  \begin{tikzpicture}[scale=.8]
  \clip (3,-2) rectangle (6,-1);
   \filldraw (3.5,-1.5) circle (2.5pt);
   \filldraw (5.5,-1.5) circle (2.5pt);
 \end{tikzpicture}}  & 2-rigid (*) & A &
  \raisebox{-0.5\height}{ 
  	\begin{tikzpicture}[scale=.8]
  	\clip (3,-2) rectangle (6,-1);
  	\filldraw (3.5,-1.5) circle (2.5pt);
  	\filldraw (5.5,-1.5) circle (2.5pt);
  	\draw [line width=1.1pt] (3.5,-1.5)--(5.5,-1.5);
  	\end{tikzpicture}}  & 2-rigid (*) & B \\
  \hline
   \raisebox{-0.5\height}{ 
  	\begin{tikzpicture}[scale=.8]
  	\clip (3,-3.5) rectangle (6,-1);
  	\filldraw (3.5,-3) circle (2.5pt);
  	\filldraw (4.5,-1.5) circle (2.5pt);
  	\filldraw (5.5,-3) circle (2.5pt);
  	\end{tikzpicture}}  & non $2$-rigid & D &
  \raisebox{-0.5\height}{ 
  	\begin{tikzpicture}[scale=.8]
  	\clip (3,-3.5) rectangle (6,-1);
  	\filldraw (3.5,-3) circle (2.5pt);
  	\filldraw (4.5,-1.5) circle (2.5pt);
  	\filldraw (5.5,-3) circle (2.5pt);
  \draw [line width=1.1pt] (3.5,-3)--(5.5,-3);
  	\end{tikzpicture}}  & $2$-rigid & A \\
   \hline
  \raisebox{-0.5\height}{ 
  \begin{tikzpicture}[scale=.8]
   \clip (3,-3.5) rectangle (6,-1);
   \filldraw (3.5,-3) circle (2.5pt);
   \filldraw (4.5,-1.5) circle (2.5pt);
   \filldraw (5.5,-3) circle (2.5pt);
   \draw [line width=1.1pt] (3.5,-3)--(4.5,-1.5);
   \draw [line width=1.1pt] (3.5,-3)--(5.5,-3);
 \end{tikzpicture}}  & $2$-rigid & A &
 \raisebox{-0.5\height}{ 
 	\begin{tikzpicture}[scale=.8]
 	\clip (3,-3.5) rectangle (6,-1);
 	\filldraw (3.5,-3) circle (2.5pt);
 	\filldraw (4.5,-1.5) circle (2.5pt);
 	\filldraw (5.5,-3) circle (2.5pt);
 	\draw [line width=1.1pt] (3.5,-3)--(4.5,-1.5);
 	\draw [line width=1.1pt] (4.5,-1.5)--(5.5,-3);
 	\draw [line width=1.1pt] (3.5,-3)--(5.5,-3);
 	\end{tikzpicture}}  & $2$-rigid & B \\
 \hline
\end{tabular}
\end{center}

\

Those marked with (*) are the only 3-rigid ones 
(see Remark \ref{rmk:3-rigid}).

\vskip1cm

\newpage

 \begin{center}
  Graph Lie algebras with 4 vertices \\
  \ \\
 \begin{tabular}{|c|c|c||c|c|c|}
 	\hline
 \raisebox{-0.5\height}{ 
 	\begin{tikzpicture}[scale=.8]
 	\clip (3,-3.5) rectangle (5.5,-1);
 	\filldraw (3.5,-1.5) circle (2.5pt);
 	\filldraw (5,-1.5) circle (2.5pt);
 	\filldraw (3.5,-3) circle (2.5pt);
 	\filldraw (5,-3) circle (2.5pt);
 	\end{tikzpicture}}  & non $2$-rigid & D &
\raisebox{-0.5\height}{ 
	\begin{tikzpicture}[scale=.8]
	\clip (3,-3.5) rectangle (5.5,-1);
	\filldraw (3.5,-1.5) circle (2.5pt);
	\filldraw (5,-1.5) circle (2.5pt);
	\filldraw (3.5,-3) circle (2.5pt);
	\filldraw (5,-3) circle (2.5pt);
	\draw [line width=1.1pt] (3.5,-3)--(5,-3);
	\end{tikzpicture}}  & non $2$-rigid & D \\
 \hline
\raisebox{-0.5\height}{ 
	\begin{tikzpicture}[scale=.8]
	\clip (3,-3.5) rectangle (5.5,-1);
	\filldraw (3.5,-1.5) circle (2.5pt);
	\filldraw (5,-1.5) circle (2.5pt);
	\filldraw (3.5,-3) circle (2.5pt);
	\filldraw (5,-3) circle (2.5pt);
	\draw [line width=1.1pt] (3.5,-3)--(5,-3);
    \draw [line width=1.1pt] (3.5,-3)--(3.5,-1.5);
	\end{tikzpicture}}  & non $2$-rigid & D &
 \raisebox{-0.5\height}{ 
 	\begin{tikzpicture}[scale=.8]
 	\clip (3,-3.5) rectangle (5.5,-1);
 	\filldraw (3.5,-1.5) circle (2.5pt);
 	\filldraw (5,-1.5) circle (2.5pt);
 	\filldraw (3.5,-3) circle (2.5pt);
 	\filldraw (5,-3) circle (2.5pt);
 	\draw [line width=1.1pt] (3.5,-3)--(5,-3);
 	\draw [line width=1.1pt] (3.5,-1.5)--(5,-1.5);
 	\end{tikzpicture}}  & $2$-rigid & A \\
 \hline
  \raisebox{-0.5\height}{ 
 	\begin{tikzpicture}[scale=.8]
 	\clip (3,-3.5) rectangle (5.5,-1);
 	\filldraw (3.5,-1.5) circle (2.5pt);
 	\filldraw (5,-1.5) circle (2.5pt);
 	\filldraw (3.5,-3) circle (2.5pt);
 	\filldraw (5,-3) circle (2.5pt);
 	\draw [line width=1.1pt] (3.5,-3)--(5,-3);
 	\draw [line width=1.1pt] (3.5,-3)--(3.5,-1.5);
 	\draw [line width=1.1pt] (3.5,-1.5)--(5,-1.5);
 	\end{tikzpicture}}  & non $2$-rigid & C &
 \raisebox{-0.5\height}{ 
 	\begin{tikzpicture}[scale=.8]
 	\clip (3,-3.5) rectangle (5.5,-1);
 	\filldraw (3.5,-1.5) circle (2.5pt);
 	\filldraw (5,-1.5) circle (2.5pt);
 	\filldraw (3.5,-3) circle (2.5pt);
 	\filldraw (5,-3) circle (2.5pt);
 	\draw [line width=1.1pt] (3.5,-3)--(5,-3);
 	\draw [line width=1.1pt] (3.5,-3)--(3.5,-1.5);
 	\draw [line width=1.1pt] (3.5,-3)--(5,-1.5);
 	\end{tikzpicture}}  & non $2$-rigid & C \\
\hline
 \raisebox{-0.5\height}{ 
 	\begin{tikzpicture}[scale=.8]
 	\clip (3,-3.5) rectangle (5.5,-1);
 	\filldraw (3.5,-1.5) circle (2.5pt);
 	\filldraw (5,-1.5) circle (2.5pt);
 	\filldraw (3.5,-3) circle (2.5pt);
 	\filldraw (5,-3) circle (2.5pt);
 	\draw [line width=1.1pt] (3.5,-3)--(5,-3);
 	\draw [line width=1.1pt] (3.5,-3)--(3.5,-1.5);
 	\draw [line width=1.1pt] (3.5,-1.5)--(5,-3);
 	\end{tikzpicture}}  & non $2$-rigid & C &
 \raisebox{-0.5\height}{ 
 	\begin{tikzpicture}[scale=.8]
 	\clip (3,-3.5) rectangle (5.5,-1);
 	\filldraw (3.5,-1.5) circle (2.5pt);
 	\filldraw (5,-1.5) circle (2.5pt);
 	\filldraw (3.5,-3) circle (2.5pt);
 	\filldraw (5,-3) circle (2.5pt);
 	\draw [line width=1.1pt] (3.5,-3)--(5,-3);
 	\draw [line width=1.1pt] (3.5,-3)--(3.5,-1.5);
 	\draw [line width=1.1pt] (3.5,-1.5)--(5,-3);
 	\draw [line width=1.1pt] (3.5,-3)--(5,-1.5);
 	\end{tikzpicture}}  & non $2$-rigid & C \\
 \hline
 \raisebox{-0.5\height}{ 
 	\begin{tikzpicture}[scale=.8]
 	\clip (3,-3.5) rectangle (5.5,-1);
 	\filldraw (3.5,-1.5) circle (2.5pt);
 	\filldraw (5,-1.5) circle (2.5pt);
 	\filldraw (3.5,-3) circle (2.5pt);
 	\filldraw (5,-3) circle (2.5pt);
 	\draw [line width=1.1pt] (3.5,-3)--(5,-3);
 	\draw [line width=1.1pt] (3.5,-3)--(3.5,-1.5);
 	\draw [line width=1.1pt] (3.5,-1.5)--(5,-1.5);
 		\draw [line width=1.1pt] (5,-1.5)--(5,-3);
 	\end{tikzpicture}}  & $2$-rigid & A &
 \raisebox{-0.5\height}{ 
 	\begin{tikzpicture}[scale=.8]
 	\clip (3,-3.5) rectangle (5.5,-1);
 	\filldraw (3.5,-1.5) circle (2.5pt);
 	\filldraw (5,-1.5) circle (2.5pt);
 	\filldraw (3.5,-3) circle (2.5pt);
 	\filldraw (5,-3) circle (2.5pt);
 	\draw [line width=1.1pt] (3.5,-3)--(5,-3);
 	\draw [line width=1.1pt] (3.5,-3)--(3.5,-1.5);
 	\draw [line width=1.1pt] (3.5,-1.5)--(5,-1.5);
 	\draw [line width=1.1pt] (5,-1.5)--(5,-3);
 		\draw [line width=1.1pt] (3.5,-3)--(5,-1.5);
 	\end{tikzpicture}}  & non $2$-rigid & C \\
 \hline
 \raisebox{-0.5\height}{ 
 	\begin{tikzpicture}[scale=.8]
 	\clip (3,-3.5) rectangle (5.5,-1);
 	\filldraw (3.5,-1.5) circle (2.5pt);
 	\filldraw (5,-1.5) circle (2.5pt);
 	\filldraw (3.5,-3) circle (2.5pt);
 	\filldraw (5,-3) circle (2.5pt);
 	\draw [line width=1.1pt] (3.5,-3)--(5,-3);
 	\draw [line width=1.1pt] (3.5,-3)--(3.5,-1.5);
 	\draw [line width=1.1pt] (3.5,-1.5)--(5,-1.5);
 	\draw [line width=1.1pt] (5,-1.5)--(5,-3);
 	\draw [line width=1.1pt] (3.5,-3)--(5,-1.5);
 	\draw [line width=1.1pt] (5,-3)--(3.5,-1.5);
 	\end{tikzpicture}}  & $2$-rigid & B & & & \\
 \hline
 \end{tabular}
\end{center}

\begin{remark}
It is worth mentioning that all degenerations that occur among all graph Lie algebras are presented in \cite{AAA} and also follow
from a result in \cite{GO} about the degeneration
of a Lie algebra to any of its quotients plus a suitable abelian factor.
\end{remark}

We come now to the general case.

The Lie algebra associated to the complete graph of $m$ vertices,
$K_m$, is the free 2-step nilpotent Lie algebra on $m$ generators.
We know that it is 2-rigid and not 3-rigid \cite{BT}.
Out main result is that besides this family, 
the only 2-rigid ones have been already listed in the above tables. 

\begin{theorem}
If $G$ is a non-complete graph with $|V|>4$, then $\g_G$ is not $2$-rigid.
\end{theorem}

\begin{proof}
Let $G$ be a non-complete graph with $|V|>4$.
For each $v_i\in V$ let $A_i$ be the subset of $A$ formed by all edges $a_{jk}$ with $j=i$ or $k=i$.

We assume, for the absurd, that $\g_G$ is $2$-rigid. 
We notice the following facts:
\begin{enumerate}
\item Since $g_G$ is $2$-rigid, then $\g_G$ has no abelian factor (Theorem C). Then $\C^n=\v\oplus\z=\langle V\rangle\oplus\langle A\rangle$ and $\z=[\v,\v]$.
\item Since $G$ is a non-complete graph, there exist vertices $v_1$ and $v_2$, such that $a_{1,2}\notin A$. That is $[v_1,v_2]=0$.
\item From the Theorem D, items (1) and (2) above and the fact that $\g_G$ is $2$-rigid, it follows that $\langle [v_1,\g]\cup[v_2,\g]\rangle=\z$. Therefore $A_1\cup A_2=A$.
\end{enumerate}

Let $v_3$, $v_4$ and $v_5$ be three different vertices different from $v_1$ and $v_2$ (recall that $|V|>4$), we have that:
\begin{enumerate}
\setcounter{enumi}{3}
\item Since $A=A_1\cup A_2$, then $[v_3,v_4]=0$.
\item Since $\emptyset\neq A_5\subseteq A_1\cup A_2$ ($v_5$ is not an abelian factor and item (3)), we can assume that $a_{15}\in A$.
\end{enumerate}
 
 Now, the hypotheses of Theorem D are satisfied by taking $v=v_3$, $w=v_4$ and $z=a_{15}$ and therefore $\g_G$ is not $2$-rigid, wich is an absurd. Thus $\g_G$ is not $2$-rigid.
\end{proof}

We can now state the following classification result.

\begin{theorem}
$\g_G$ is $2$-rigid if and only if $G$ is a complete graph $K_m$ 
or is one of the following 5 graphs:

\begin{center}
 \begin{tikzpicture}
   \filldraw (-2,0.5) circle (2.5pt); 
   \filldraw (-1,0.5) circle (2.5pt); 
   
   \filldraw (1.5,1) circle (2.5pt);
   \filldraw (1,0) circle (2.5pt);
   \filldraw (2,0) circle (2.5pt); 
   \draw [line width=1.1pt] (1,0)--(2,0);
   
   \filldraw (4.5,1) circle (2.5pt);
   \filldraw (4,0) circle (2.5pt);
   \filldraw (5,0) circle (2.5pt); 
   \draw [line width=1.1pt] (4,0)--(5,0);
   \draw [line width=1.1pt] (4,0)--(4.5,1);

   \filldraw (7,1) circle (2.5pt);
   \filldraw (8,1) circle (2.5pt);
    \filldraw (7,0) circle (2.5pt);
   \filldraw (8,0) circle (2.5pt);
   \draw [line width=1.1pt] (7,1)--(8,1);
   \draw [line width=1.1pt] (7,0)--(8,0);
   
   \filldraw (10,1) circle (2.5pt);
   \filldraw (11,1) circle (2.5pt);
 \filldraw (10,0) circle (2.5pt);
\filldraw (11,0) circle (2.5pt);
   \draw [line width=1.1pt] (10,1)--(11,1);
 \draw [line width=1.1pt] (10,0)--(11,0);
\draw [line width=1.1pt] (10,1)--(10,0);
\draw [line width=1.1pt] (11,0)--(11,1);
 \end{tikzpicture}
\end{center}

\ 

\end{theorem}


\end{document}